\documentclass[11pt]{article}
\vspace{10mm}
\usepackage{amsmath, amssymb, amscd, amsthm, amsfonts}
\usepackage{hyperref}
\usepackage{xcolor}

\usepackage{titlesec}
\titlespacing*{\paragraph}
  {0pt}   
  {2ex}  
  {2ex}

\theoremstyle{definition}
\newtheorem{theorem}{Theorem}[section]
\newtheorem{thm}[theorem]{Theorem}
\newtheorem{lemma}[theorem]{Lemma}

\newtheorem{prop}[theorem]{Proposition}

\newtheorem{example}[theorem]{Example}
\newtheorem{corollary}[theorem]{Corollary}
\newtheorem{question}[theorem]{Question}

\let\ti\widetilde
\def\cl{{\fam0 cl}}
\def\mZ{\mathbb{Z}}

\title{Bipartite holes, degree sums and Hamilton cycles}

\author{M.~N.~Ellingham\footnote{Department of Mathematics, Vanderbilt University, Nashville, TN, USA (\href{mailto:mark.ellingham@vanderbilt.edu}{mark.ellingham@vanderbilt.edu}). Supported by Simons Foundation award MPS-TSM-00002760.} \and Yixuan Huang\footnote{Department of Mathematics, Vanderbilt University, Nashville, TN, USA (\href{mailto:yixuan.huang.2@vanderbilt.edu}{yixuan.huang.2@vanderbilt.edu}).} \and Bing Wei\footnote{Department of Mathematics, University of Mississippi, University, MS, USA (\href{mailto:bwei@olemiss.edu}{bwei@olemiss.edu}). Supported in part by an SEC Faculty Travel Grant.}}

\date{}

\begin{document}

\maketitle

\begin{abstract}
    The {\em bipartite-hole-number} of a graph $G$, denoted as $\ti\alpha(G)$, is the minimum number $k$ such that there exist integers $a$ and $b$ with $a + b = k+1$ such that for any two disjoint sets $A, B \subseteq V(G)$, there is an edge between $A$ and $B$.
    McDiarmid and Yolov initiated research on bipartite holes by extending Dirac's classical theorem on minimum degree and Hamiltonian cycles.
    They showed that a graph on at least three vertices with $\delta(G) \ge \ti\alpha(G)$ is Hamiltonian.
    Later, Dragani\'c, Munh\'a Correia and Sudakov proved that $\delta\ge \ti\alpha(G)$ implies that $G$ is pancyclic, unless $G = K_{\frac n2, \frac n2}$.  This extended the result of McDiarmid and Yolov and generalized a theorem of Bondy on pancyclicity.
    
    In this paper, we show that a $2$-connected graph $G$ is Hamiltonian if $\sigma_2(G) \ge 2 \ti\alpha(G) - 1$, and that a connected graph $G$ contains a cycle through all vertices of degree at least $\ti\alpha(G)$.
    Both results extended McDiarmid and Yolov's result.
    As a step toward proving pancyclicity, we show that if an $n$-vertex graph $G$ satisfies $\sigma_2(G) \ge 2 \ti\alpha(G) - 1$, then it either contains a triangle or it is $K_{\frac n2, \frac n2}$.
    Finally, we discuss the relationship between connectivity and the bipartite hole number.
\end{abstract}

\section{Introduction}

    Throughout this paper, all graphs are simple, i.e., have no loops or multiple edges. 
    We use the notation and terminology of Bondy and Murty \cite{BondyMurty}.

    Hamilton cycles are spanning cycles of a graph.
    The problem of determining whether a graph has a Hamilton cycle, as one of the central topics in graph theory, is one of Karp's 21 NP-complete problems \cite{Karp1972}.
    The classic result by Dirac~\cite{Dirac} states that a graph on $n \ge 3$ vertices has a Hamilton cycle if $\delta(G) \ge n/2$, where $\delta(G)$ is the minimum degree of $G$.
    Another result by Ore~\cite{Ore} which generalizes Dirac's theorem states that a graph has a Hamilton cycle if $\sigma_2(G) \ge n$, where $\sigma_2(G) = \min\{d(x) + d(y): xy \notin E(G)\}$.
    
    An $(s,t)$-{\em bipartite-hole} of $G$ is a pair of disjoint sets $S$ and $T$ such that $|S| = s$, $|T| = t$ and there is no edge between $S$ and $T$.
    The {\em bipartite-hole-number} $\ti\alpha(G)$ is defined by the minimum number $k$ such that $G$ contains no $(s,t)$-bipartite-hole for positive integers $s,t$ with $k + 1 = s + t$.
    In other words, for any positive integers $a,b$ with $a + b = \ti \alpha(G)$, $G$ contains an $(a,b)$-bipartite-hole.
    It is easy to see that $\ti\alpha(G) \geq \alpha(G)$, where $\alpha(G)$ is the independence number of $G$.
    As another generalization of Dirac's theorem that arose from the investigation of the random perfect graph~\cite{MY-random-perfect-graph}, McDiarmid and Yolov proved the following theorem.

\begin{thm}[\cite{MY}]
    \label{MY}
    Let $G$ be a graph with $\delta(G) \geq \ti\alpha(G)$.
    Then $G$ is Hamiltonian.
\end{thm}

A graph on $n \ge 3$ vertices is {\em pancyclic} if it contains cycles of length $k$ for all $3 \le k \le n$.
In 1971, Bondy~\cite{Bondy1971pancyclic} proved that if $|E(G)| \ge \frac{n^2}{4}$, then $G$ is pancyclic unless $G$ is the complete balanced bipartite graph $K_{\frac n2, \frac n2}$.
Bondy's result implies that given either Dirac's or Ore's condition, the graph is pancyclic or $K_{\frac n2, \frac n2}$.
Then in 1973, Bondy proposed his famous meta conjecture~\cite{Bondy1973pancyclic}.
It asserts that any non-trivial condition which implies Hamiltonicity, also implies pancyclicity, up to a small class of exceptional graphs.
Recently, Dragani\'c, Munh\'a Correia, Sudakov extended Theorem~\ref{MY} to pancyclicity.

\begin{thm}[\cite{DMS}]
    \label{DMS}
    Let $G$ be a graph with $\delta(G) \geq \ti\alpha(G)$.
    Then $G$ is pancyclic, unless $G \cong K_{\frac n2, \frac n2}$.
\end{thm}

As an analog of Ore's result, we prove that if $G$ is $2$-connected and $\sigma_2(G) \ge 2 \ti\alpha(G) - 1$, then $G$ is Hamiltonian. 
For two complete graphs $K_a,K_b$, we denote by $K_a \oplus_1 K_b$ the graph obtained by identifying a vertex of $K_a$ and a vertex of $K_b$, that is, a {\em 1-sum} of two complete graphs.

\begin{thm}
    \label{sigma2}
    Let $G$ be a graph on $n$ vertices with $\sigma_2(G) \ge 2 \ti \alpha(G) - 1$.
    Then, $G$ is Hamiltonian, unless $G$ is a subgraph of $K_{k} \oplus_1 K_{n+1-k}$ for some $k \le \frac{n+2}4$.
\end{thm}

By Theorem~\ref{sigma2}, if $G$ is $2$-connected in addition to $\sigma_2(G) \ge 2 \ti\alpha(G) - 1$, $G$ is immediately Hamiltonian.

\begin{theorem}
    \label{2-conn}
    If $G$ is a $2$-connected graph such that $\sigma_2(G) \ge 2 \ti \alpha(G) - 1$.
    Then, $G$ is Hamiltonian.
\end{theorem}

Given a vertex subset $S$ of $G$, an $S$-{\em cycle} of $G$ is a cycle that contains $S$, and $G$ is $S$-{\em cyclable} if $G$ contains an $S$-cycle.
In particular, a $V(G)$-cycle is a Hamilton cycle.
A cycle is $X$-{\em longest} if it contains the most vertices of $X$.
We show that $G$ contains a cycle through all vertices of degree at least $\ti\alpha(G)$.

\begin{theorem}
    \label{Xcycle}
    Let $G$ be a graph and $X = \{x \, : \, d_G(x) \ge \ti\alpha(G) \}$.
    Then $G$ is $X$-cyclable.
\end{theorem}

This result generalizes McDiarmid and Yolov's Theorem~\ref{MY} and it is sharp in the sense that we cannot expect a cycle through all vertices of degree at least $\ti\alpha(G) - 1$ by the example $K_{k,k+1}$, where $\ti\alpha(K_{k,k+1}) = k+1$ and $K_{k,k+1}$ has no Hamilton cycle.

As a step toward the pancyclicity, we prove the following.

\begin{theorem}
    \label{triangle or bip}
    Let $G$ be a graph such that $\sigma_2 \ge 2 \ti\alpha(G) - 1$. 
    Then either $G$ contains a triangle or $G$ is a balanced complete bipartite graph, i.e., $G \cong K_{n,n}$.
\end{theorem}

We also investigate the relation of the bipartite-hole-number $\ti\alpha(G)$ and the connectivity $\kappa(G)$ of a graph $G$, though the connection between them is not immediately apparent from the definition.
McDiarmid and Yolov~\cite{MY} proved an inequality involving these two parameters and also the minimum degree $\delta(G)$.  Although they stated it for all graphs, in fact it holds only for non-complete graphs.

\begin{lemma}{\cite{MY}}
    \label{kappa-delta-alphatilde}
    For every non-complete graph $G$, $\kappa(G) \ge \delta(G) + 2 - \ti\alpha(G)$.
\end{lemma}

We show an intrinsic relationship between $\kappa(G)$ and $\ti\alpha(G)$ using the following propositions.

\begin{prop}\label{prop: biphole}
    Let $G$ be a graph on $n$ vertices with minimum degree $\delta$. Then $$\ti\alpha(G)  =n - \max_{s \ge 1} \min_{|S| = s} |N(S)|=n - \max_{1 \le s \le n- \delta} \min_{|S| = s} |N(S)|.$$
\end{prop}

\begin{prop}\label{prop: connectivity}
    Let $G$ be a graph on $n$ vertices with minimum degree $\delta$. Then $$\kappa(G) = \min\limits_{1 \le s \le n-\delta} \min\limits_{|S| = s} |N(S)|.$$
\end{prop}

Our formulas provide a proof of the following, which was also observed (without proof) by McDiarmid and Yolov \cite[p.~280]{MY}.

\begin{corollary}
    For any graph $G$ on $n$ vertices, $\kappa(G) + \ti\alpha(G) \le n$.
\end{corollary}

Throughout the paper, if a path or cycle is clear from the context and is denoted by $v_1,v_2,\dots,v_\ell$, then for any vertex $v = v_i$ we define $v^+ = v_{i+1}$ and $v^- = v_{i-1}$, where the indices are modulo $\ell$ if the underlying graph is a cycle.
For a vertex subset $S$, we define $S^+ = \{v^+: v \in S\}$ and $S^- =\{v^- : v \in S\}$.

Given two vertex subsets $A,B$ of a graph $G$ that are not necessarily disjoint, we use $E_G[A,B]$ to denote the set of edges with one end in $A$ and the other end in $B$, or simply $E[A,B]$ if the underlying graph is clear from the context.

\section{A degree sum result for Hamilton cycles}

\begin{proof}[Proof of Theorem~\ref{sigma2}]
    Let $G$ be a non-Hamiltonian graph such that $\sigma_2(G) \ge 2 \ti\alpha(G) - 1$ with maximum number of edges.
    Since $\sigma_2$ is non-decreasing and $\ti\alpha$ is non-increasing when adding edges, the condition $\sigma_2 \geq 2 \ti \alpha - 1$ holds when adding edges to $G$.
    Therefore, $G$ contains a Hamilton path.
    Let $P = v_1 v_2 \dots v_n$ be a Hamilton path of $G$ such that $d(v_1) \le d(v_n)$, $d(v_1)$ is maximum possible, and subject to that, $d(v_n)$ is maximized.
    Note that $v_n \in X$, i.e., $d(v_n) \ge \ti\alpha(G)$, otherwise $v_1,v_n \in Y$ and $P + v_1v_n$ is a Hamilton cycle.

    Now we partition the vertex set into two parts.
    Let $X =\{ v \,:\, d(v) \ge \ti\alpha(G)\}$ and $Y = \{ v \;:\, d(v) \le \ti\alpha(G) - 1\}$.
    Note that given $\sigma_2 \geq 2 \ti\alpha - 1$, $G[Y]$ is a clique.
    
    \paragraph{Case 1.}
    $v_1 \in X$, i.e., $d(v_n) \ge d(v_1) \ge \ti\alpha$.

    Let $\ti\alpha + 1 = s + t$, where $s \le t$ are positive integers such that $G$ contains no $(s,t)$-bipartite holes.
    Let $k$ be the minimum index such that $\{v_1,\dots,v_k\}$ contains exactly $s$ neighbors of $v_1$, then $v_k \in N(v_1)$.
    
    Let $S_1 = (P[v_1,v_k] \cap N(v_1))^-$, which has size $s$, and $T_1 = (P[v_k,v_n] \cap N(v_n))^+$.
    If there exists $v_iv_j \in E(G)$ for $v_i \in S_1$ and $v_j \in T_1$, there is a Hamilton cycle $P[v_1,v_i] + v_iv_j + P[v_j,v_n] + v_nv_{j-1} + P[v_{j-1},v_{i+1}] + v_{i+1}v_1$.
    Therefore, there is no edge between $S_1$ and $T_1$, and $|T_1| < t$, otherwise there is an $(s,t)$-bipartite-hole.

    Let $S_2 = (P[v_k,v_n] \cap N(v_1))^+$, which has size at least $t$, and $T_2 = (P[v_1,v_k) \cap N(v_n))^+$.
    If there exists $v_s v_t \in E(G)$ for $v_s \in S_2$ and $v_t \in T_2$, there is a Hamilton cycle $P[v_1,v_{t-1}] + v_{t-1}v_n + P[v_n,v_s] + v_sv_t + P[v_t,v_{s-1}] + v_{s-1}v_1$.
    Therefore, there is no edge between $S_2$ and $T_2$, and thus $|T_2| < s$, otherwise there is an $(s,t)$-bipartite-hole.

    We deduce a contradiction since $\ti\alpha \le d_G(v_n) =  |T_1| + |T_2| \le s - 1 + t - 1 = \ti \alpha - 1$.

    \paragraph{Case 2.} $v_1 \in Y$, i.e., $d_G(v_1) \le \ti\alpha - 1$.
    
    For each $v_i \in N(v_1)$, note that the path $Q = v_{i-1} \dots v_2 v_1 v_i v_{i+1} \dots v_n$ is a Hamilton path.
    By the maximality of $d_G(v_1)$, $d_G(v_{i-1}) \le d_G(v_1) \le \ti\alpha(G) - 1$.
    Hence, $v_{i-1} \in Y$ and since $G[Y]$ is a clique, $v_{i-1} \in N(v_1)$.
    Therefore, $N(v_1)= \{v_2,v_3,\dots,v_{d+1}\}$, where $d = d_G(v_1)$, and $\{v_1,v_2,\dots,v_d\} \subseteq Y \subseteq N[v_1] =\{v_1,v_2,\dots,v_{d+1}\}$.

    {\em Case 2.1.} $v_{d+1}$ cannot separate $\{v_1,v_2,\dots,v_d\}$ and $\{v_{d+2},\dots,v_n\}$.
    
    There exists $v_iv_j \in E(G)$, where $2 \le i \le d$ and $d+2 \le j \le n$.
    Consider the Hamilton path $R = v_{j-1}\dots v_{d+1} \dots v_{i+1} v_1 \dots v_i v_j  \dots v_n$.
    Since $R$ has two ends $v_{j-1}$ and $v_n$, then $d(v_{j-1}) \le d(v_1) < \ti\alpha$ by the maximality of $d(v_1)$.
    Hence, $v_{j-1} \in Y \subseteq \{v_1,v_2,\dots,v_{k+1}\}$, and thus $j\le d+2$.
    Therefore, $j = d+2$.
    Note that $v_{d+1}v_d \dots v_{i+1} v_1 v_2 \dots v_i v_{d+2} v_{d+3} \dots v_n$ is a Hamilton path with ends $v_{d+1}$ and $v_n$.
    By the maximality of $d(v_1)$, $d(v_{d+1}) \le d(v_1) = d$.
    Since $\{v_1,v_2,\dots,v_d\} \cup \{v_{d+2}\} \subseteq N(v_{d+1})$, $d(v_{d+1}) \ge d+1$, a contradiction.

    {\em Case 2.2.} $v_{d+1}$ separates $\{v_1,\dots,v_d\}$ and $\{v_{d+2},\dots,v_n\}$.

    Since $v_{d+1}$ is still a cut vertex and thus $G$ is not Hamiltonian whenever we add an edge between any two vertices $v_i,v_j$ for $i,j \le d+1$ or $i,j \ge d+1$,
    by the maximality of $|E(G)|$, both sets $\{v_1,\dots,v_d,v_{d+1}\}$ and $\{v_{d+1},\dots,v_n\}$ induce cliques, and hence, $d(v_n) = n-d-1 \ge \ti\alpha(G) \ge d+1$, $n \ge 2d + 2$, and $\sigma_2(G) = d + (n-d-1) = n - 1$.
    To determine $\ti\alpha(G)$, for every pair of vertex sets $S,T$ such that $(S,T)$ is $(s,t)$-bipartite-hole of $G$, without loss of generality, we have $S\subseteq \{v_1,\dots,v_d\}$ and $T \subseteq \{v_{d+2},\dots,v_n\}$.
    Therefore, $G$ has no $(1,n-d)$-bipartite-hole or $(d+1,d+1)$-hole, and $G$ contains $(s,t)$-bipartite hole for any $s + t \le \min\{n-d,2d+1\}$, which implies that
    $\ti\alpha(G) = \min\{n-d, 2d + 1\}$.
    If $\ti\alpha(G) = n-d$, then $\sigma_2(G) = n-1 \ge 2 \ti\alpha(G) - 1 = 2n-2d-1$ and thus $2d \ge n \ge 2d + 2$, a contradiction.
    Therefore, $\ti\alpha(G) = 2d+1$ and $\sigma_2(G) \ge 2 \ti\alpha(G) - 1$ implies that $n - 1 \ge 4d+1$, that is, $d \le \frac{n-2}4$.
    
    \paragraph{}In conclusion, for every $n$-vertex non-Hamiltonian $G$ such that $\sigma_2(G) \ge 2 \ti\alpha(G) - 1$, $G$ is a subgraph of $K_k \oplus_1 K_{n+1-k}$ for some $k \le \frac{n+2}4$.
\end{proof}

\section{Cyclability and triangles}

\begin{proof}[Proof of Theorem~\ref{Xcycle}]
    First, $X$ belong to a block of $G$.
    Indeed, if there are two vertices $x,y \in X$ in different blocks, and a vertex $v \neq x,y$ separates $x$ and $y$, then there is no edge between $N[x]\setminus \{v\}$ and $N[y] \setminus \{v\}$.
    Since $|N[x] \setminus \{v\} | = d_G(x) \geq \ti\alpha(G)$ and $|N[y] \setminus \{v\}| \ge \ti\alpha(G)$, there is an $(s,t)$-bipartite-hole for any positive integers $s,t$ with $s + t = \ti\alpha(G) + 1$.
    Therefore, $X$ belong to a block $B$ of $G$.

    Let $C$ be a cycle of $G$ such that $C$ is $X$-longest and subject to it, $C$ is longest.
    Then, there is no cycle $C'$ such that $V(C) \subsetneq V(C')$.
    Assume that $X \not\subseteq V(C)$.
    Let $H$ be a component of $G - V(C)$ intersecting $X$, and $T = N_G(H) \subseteq V(C)$.
    Consider $T^+$.
    
    {\bf Claim 1.}
    $E[T^+, H] = \emptyset$. That is, $T \cap T^+ = \emptyset$.
    
    {\em Proof.}
    If not, there are two consecutive vertices $v,v^+$ on $C$ such that $v,v^+ \in T$.
    Let $u \in H \cap N(v)$ and $w \in H  \cap N(v^+)$.
    Since $H$ is connected, there is a $uw$-path $P$ in $H$, then there is a longer cycle $C[v^+,v] + vu + P + wv^+$ that contains $V(C)$, a contradiction.
    \qed
    
    {\bf Claim 2.}
    $E[T^+, T^+] = \emptyset$. That is, $G[T^+]$ is empty.
    
    {\em Proof.}
    If not, there are two vertices $u,v \in T$ such that $u^+ v^+ \in E(G)$.
    Let $y \in H \cap N(u)$ and $z \in H \cap N(v)$.
    Since $H$ is connected, there is a $yz$-path $P$ in $H$, then there is a longer cycle $C[u^+,v] + vz + P + C[v^+,u] + u^+v^+$ that contains $V(C)$, a contradiction.
    \qed

    Let $s \le t$ be positive integers such that $s + t = \ti\alpha(G)+ 1$ and $G$ contains no $(s,t)$-bipartite-holes.
    Since $H \cap X \neq \emptyset$, let $x \in H \cap X$. 
    Then $N[x] \subseteq H \cup T$, $|H \cup T^+| = |H \cup T| \geq d_G(x) +1 \ge \ti\alpha(G) + 1= s + t$.
    By Claims 1 and 2, $|T| = |T^+| < s$, otherwise, let $A$ be a subset of $T^+$ of size $s$, and $B$ be a subset of $(H \cup T^+) \setminus A$ of size $t$, $(A,B)$ is an $(s,t)$-bipartite-hole, a contradiction.
    Therefore, $|H| \ge s + t - |T| \geq t + 1$.

    Since $H$ is a component of $G - V(C)$ and $T= N(H) \cap V(C)$, $E[H,G \setminus (H \cup T)] = \emptyset$ and thus $|G \setminus (H \cup T)| < s$, otherwise $(G \setminus (H \cup T), H)$ is a $(s',t')$-bipartite-hole with $s' \ge s$ and $t' \ge t$.
    Therefore, $|G \setminus H| \le s-1 + |T| \le 2s - 2 \le s + t - 2 = \ti\alpha(G) - 1 < \ti\alpha(G)$.
    Since for any $v \in G \setminus (H \cup T)$, $N[v] \subseteq G \setminus H$, then $d_G(v) < \ti\alpha(G)$ and $v \in Y$, where $Y = V(G) \setminus X = \{y \in V(G) \,:\, d_G(y) < \ti\alpha(G)\}$.
    Hence, $X \subseteq H \cup T$ and for any $x \in T \cap X$, $|N(x) \cap H| \ge d_G(x) - (|G \setminus H| - 1) \ge \ti\alpha(G) - (\ti\alpha(G) - 2) = 2$.
    
    Let $x_0 \in H \cap X$, and $X \cap C = \{x_1,x_2,\dots,x_k\}$, where $x_1,\dots,x_k$ occur in the order of $C$.
    Since $X$ belong to a block of $G$, there is a cycle containing every pair of vertices in $X$, hence we may assume $k \ge 2$.

    Let $B_1 ,\dots,B_m$ be the blocks of $H$ and $x_0 \in B_1$.
    Note that $B_1 \subseteq B$ since all vertices of $X$ belong to the same block $B$ and $x_0 \in B_1 \cap X$.
    Let $B_1^\circ$ be the set of vertices in $B_1$ that are not cut vertices of $H$.
    Let $D_1,\dots,D_l$ be the components of $H \setminus B_1^\circ $, and $v_1,\dots,v_j$ be the vertices of $B_1^\circ$.
    Hence, $\{D_1,\dots,D_l,\{v_1\},\dots,\{v_j\}\}$ is a partition of $H$. 
    % If H is 2-connected, then we just divide H into single vertices.
    Note that there exists a $uv$-path through $x_0$ between any two vertices $u,v$ in different parts.

    {\bf Case 1.} The vertices of $\cup_{i=1}^k N(x_i) \cap H$ do not belong to a same part of $H$.

    There exists an index $i$ (modulo $k$) such that there exist $a_i \in N(x_i) \cap H$ and $a_{i+1} \in N(x_{i+1}) \cap H$ in different partitions.
    Let $P$ be an $a_ia_{i+1}$-path containing $x_0$, we obtain an $X$-longer cycle $a_{i+1}x_{i+1} + C[x_{i+1}, x_i] + x_i a_i + P$, a contradiction.

    {\bf Case 2.} The vertices of $\cup_{i=1}^k N(x_i) \cap H$ belong to the same part of $H$.

    Since $|N(x) \cap H | \ge 2$ for each $x \in X \cap T$, we may assume that all the vertices of $\cup_{i=1}^k N(x_i) \cap H$ belong to the component $D_1$.
    Let $v_1 \in B_1 \cap D_1$ be the cut vertex that separates $B_1$ and $D_1$.
    Since $B_1 - v_1 \subseteq B$ and $G - V(H)$, which contains $V(C)$, intersects $B$, by the $2$-connectivity of $B$, $v_1$ cannot separate $B_1 - v_1$ and $G-V(H)$, and there is an edge $vw$ where $v \in H \setminus V(D_1)$ and $w \in T = N_G(H) \subseteq V(C)$.
    Suppose $w \in V(C[x_i, x_{i+1}])$.
    Let $a_i \in N(x_i) \cap H \subseteq D_1$. Since $a_i$ and $v$ belong to different parts of $H$, there is an $a_iv$-path $P$ in $H$ through $x_0$.
    Therefore, $P + vw + C[w,x_i] + x_ia_i$ is cycle that contains $X \cap V(C)$ and $x_0$, which is $X$-longer, a contradiction.
\end{proof}

\begin{proof}[Proof of Theorem~\ref{triangle or bip}]
    Suppose for the sake of contradiction that $G$ is non-bipartite and triangle-free. 
    Then $G$ contains a cycle of odd length at least $5$ that is an induced subgraph.
    
    Let $X = \{v\,:\, d_G(v) \ge \ti\alpha(G)\}$ and $Y = \{v \,:\, d_G(v) \le \ti\alpha(G) - 1\}$.
    Since $d_G(u) + d_G(v) \le 2\ti\alpha(G) - 2 < \sigma_2$ for any two vertices $u,v \in Y$, $G[Y]$ is a clique.
    Hence, $|Y| \le 2$, that is, all but at most two vertices have degree at least $\ti\alpha(G)$.

    Let $l \ge 5$ be the minimum odd number such that $G$ contains an induced cycle $C_l$ of length $l$.

    {\bf Claim 3.} There do not exist $x,y,z \in X \cap V(C_l)$ such that $yz \in E(G)$ and $xy,xz \notin E(G)$.

    {\em Proof.}
    Let $x,y,z \in X \cap V(C_l)$ be such that $yz \in E(G)$ and $xy,xz \notin E(G)$.
    Suppose that $s \le t$ are two positive integers such that $s + t = \ti\alpha(G) + 1$ and $G$ contains no $(s,t)$-bipartite-hole.
    Since $G$ is triangle-free, for any vertex $v \in G$, $N(v)$ induces an empty graph.

    Note that $|N(y) \cup \{x\}| \ge d_G(y) + 1 \ge \ti\alpha(G) + 1 = s + t$.
    If $|N(y) \cap N(x)| < t$, then let $T$ be a subset of $N(y) \cup \{x\}$ of size $t$ that contains $(N(y) \cap N(x)) \cup \{x\}$ and $S$ be a subset of $(N(y) \cup \{x\}) \setminus T$ of size $s$.
    Then, $(S,T)$ is an $(s,t)$-bipartite-hole in $G$, a contradiction.
    Therefore, $|N(y) \cap N(x)| \ge t$ and similarly, $|N(z) \cap N(x)| \ge t$.

    Since $yz \in E(G)$, we have $N(y) \cap N(z) = \emptyset$, otherwise, for any common neighbor $v$ of $y$ and $z$, $y,z,v$ is a triangle.
    Therefore, $d_G(x) \ge |N(x) \cap N(y)| + |N(x) \cap N(z)| \ge t + t \ge s + t$.
    For two disjoint subsets $S,T$ of $N(x)$ of sizes $s$ and $t$ respectively, $(S,T)$ is an $(s,t)$-bipartite-hole, a contradiction.
    \qed
    
    If $Y \cap V(C_l) = \emptyset$, there exist $x,y,z \in X \cap V(C_l)$ with $yz \in E(G)$ and $xy, xz \notin E(G)$, which is a contradiction.
    If $Y \cap V(C_l) \neq \emptyset$ and $l >5$, then for a vertex $x \in V(C_l) \cap X$ that is adjacent to a vertex in $Y \cap V(C_l)$, there are at least four vertices on $C_l$ that are not adjacent to $x$, where there are at least two adjacent vertices $y,z$ in $X$.
    Therefore, $l = 5$.
    
    For every induced cycle $C_5$ of length $5$, $C_5$ has two consecutive vertices in $Y$, otherwise there exist $x,y,z \in X \cap V(C_5)$ as the Claim states.
    Let $C_5 = a,b,c,y_1,y_2,a$ where $a,b,c \in X$ and $y_1,y_1 \in Y$.
    Since $G$ is triangle-free, $N(a) \cap N(b) = N(b) \cap N(c) = \emptyset$.
    Note that $N(a)\setminus V(C_5)$ and $N(c) \setminus V(C_5)$ are contained in $X$.
    If there exists an edge between $N(a) \setminus V(C_5)$ and $N(c) \setminus V(C_5)$, then we can find two vertices $x_1,x_2\in X$ such that $a,b,c,x_1,x_2$ is a cycle $C_5'$ of length $5$ and thus it is induced.
    Then, there exist $x,y,z$ in $C_5'$ as the Claim states, a contradiction.
    Therefore, $E[N(a) \setminus V(C_5), N(c) \setminus V(C_5)] = \emptyset$, and thus $\min \{|N(a) \setminus V(C_5)|, |N(c) \setminus V(C_5) |\} <t$, otherwise, we can find an $(s,t)$-bipartite-hole by picking $s$ vertices from $N(a) \setminus V(C_5)$ and $t$ vertices from $N(c) \setminus V(C_5)$. 
    Since $d_G(a), d_G(c) \ge \ti\alpha(G)$, then $s + t -1 = \ti\alpha(G) - 2 \le \min \{|N(a) \setminus V(C_5)|, |N(c) \setminus V(C_5)|\} \le t-1$, and thus $s \le 0$, a contradiction.

    Therefore, we conclude that either $G$ is bipartite or $G$ contains a triangle.

    If $G$ is bipartite, suppose that $V(G) = A \sqcup B$ is a bipartition of $G$, where $|A| = a$, $|B| = b$ and $a \le b$.
    Since $B$ induces an empty graph, we have $\ti\alpha(G) \ge \alpha(G) \ge b$ and thus $\sigma_2(G) \ge 2b -1$.
    For two distinct vertices $u,v \in B$, $2a \ge d_G(u) + d_G(v) \ge \sigma_2 \ge 2b - 1$.
    Hence, $a = b$.
    If there exists $u \in A$ and $v \in B$ that are non-adjacent, then $d_G(u) + d_G(v) \le (a-1) + (b-1) = 2b-2 < \sigma_2(G)$, a contradiction.
    Therefore, $G$ is isomorphic to $K_{a,a}$.
\end{proof}

\section{Bipartite-hole-number and connectivity}

In this section, we prove Propositions~\ref{prop: biphole} and~\ref{prop: connectivity}.

\begin{proof}[Proof of Proposition~\ref{prop: biphole}]
    For a vertex subset $S \subseteq V(G)$ with $|S| = s$, let $T = V(G) \setminus N[S]$, then $|T| = n - s - |N(S)|$ and $(S,T)$ is an $(s,t)$-bipartite-hole.
    Let $\ti \alpha_s : = \max\limits_{|S| = s} (n - s - |N(S)|)$. Then $G$ contains an $(s,\ti\alpha_s)$-bipartite-hole and no $(s,\ti\alpha_s+1)$-bipartite-hole.
    Hence $\ti\alpha(G) \le s + \ti\alpha_s$ for all $s \ge 1$.
    Taking the minimum, we obtain
    $$ \ti\alpha(G) = \min_{s \ge 1}\; (s + \ti\alpha_s).$$
    For $s \ge n - \delta(G) + 1$, $s + \ti\alpha_s \ge n-\delta(G) + 1 > 1 + \ti\alpha_1 = \max\limits_{v \in V(G)} (n - |N(v)|) = n - \delta(G)$.
    Therefore, the minimum is only achieved by integers $s \le n - \delta(G)$.  Thus,
    $$\ti\alpha(G) = \min_{s \ge 1}\; (s + \ti\alpha_s) = \min_{s \ge 1} \max_{|S| = s} (n - |N(S)|) = n - \max_{s \ge 1} \min_{|S| = s} |N(S)| = n - \max_{1 \le s \le n- \delta} \min_{|S| = s} |N(S)|.$$
\end{proof}

\begin{proof}[Proof of Theorem~\ref{prop: connectivity}]
    If $G$ is complete, then $n-\delta = 1$ and thus $\min\limits_{1 \le s \le n-\delta}\min\limits_{|S| = s} |N(S)| = \min\limits_{v \in V(G)} |N(v)| = n-1 = \kappa(G)$. Hence, we may assume that $G$ is non-complete.
    
    Let $A$ be a minimum separating set of $G$.
    Then $|A| = \kappa(G)$.
    Let $S_0$ be a component of $G \setminus A$.
    Then $N(S_0) = A$, otherwise, $N(S_0) \subsetneq A$ is a smaller separating set of $G$.
    Since $S_0$ is separated by $A$ from at least one vertex $v$, then $S_0 \cap N[v] = \emptyset$, thus $|S_0| \le n - |N[v]| < n - \delta$.
    Therefore, $\kappa(G) = |A| = |N(S_0)| \ge \min\limits_{1 \le s < n-\delta} \min\limits_{|S| = s} |N(S)|$.

    If $\min\limits_{1 \le s < n - \delta} \min\limits_{|S| = s} |N(S)|$ is achieved by a subset $S_1$ such that $|N(S_1)| < n - |S_1|$, then $N(S_1)$ is a vertex subset that separates $S_1$ and $V(G) \setminus (S_1 \cup N(S_1))$ and thus $\min\limits_{1 \le s < n-\delta} \min\limits_{|S| = s} |N(S)| = |N(S_1)| \ge \kappa(G)$.

    If $\min\limits_{1 \le s < n-\delta} \min\limits_{|S| = s} |N(S)|$ is only achieved by subsets $S_2$ such that $|N(S_2)| \ge n - |S_2|$, then for some such $S_2$ we have
    $$\min\limits_{1 \le s < n-\delta} \min\limits_{|S| = s} |N(S)| = |N(S_2)| \ge n- |S_2| \ge n - (n-\delta-1) =  \delta+1.$$
    But if $S=\{v\}$ where $v$ is a vertex of degree $\delta$, then we have $|N(S)| = \delta$, contradicting the minimum being at least $\delta+1$.  So this situation cannot happen.

    Now, suppose that $|S| = n-\delta$.  There is no $v \in V(G)\setminus(S \cup N(S))$, because then $N[v]$ would be disjoint from $S$, giving $|N[v] \cup S| \ge (\delta+1) + (n-\delta) = n+1$, which is impossible. Thus, $S \cup N(S) = V(G)$ and hence $|N(S)| = n - |S| = \delta \ge \kappa(G)$.
    Therefore, $\kappa(G) = \min\limits_{1 \le s < n-\delta} \min\limits_{|S| = s} |N(S)| = \min\limits_{1 \le s \le n-\delta} \min\limits_{|S| = s} |N(S)|$.
\end{proof}

By Proposition~\ref{prop: biphole}, we can give a complete description of the graphs with $\ti\alpha = n$.

\begin{prop}\label{prop: bh=n}
    A graph $G$ on $n$ vertices has $\ti\alpha(G) = n$ if and only if the sizes of all components of $G$, say $a_1\le a_2\le\dots\le a_r$, satisfy the inequality $a_{i} - 1 \le A_{i-1} := \sum\limits_{k=1}^{i-1}a_k$ for all $i = 1,\dots,r$, where $A_0= 0$. 
\end{prop}

\begin{proof}
By Proposition~\ref{prop: biphole}, $\ti\alpha(G) = n$ if and only if $\max_{s \ge 1} \min_{|S| = s} |N(S)| = 0$.
Note that $|N(S)| = 0$ if and only if $S$ is a union of components of $G$.
Therefore for any $s \ge 1$, there exists a vertex subset $S$ of size $s$ such that $S$ is a union of components of $G$.
Let the sizes of components of $G$ be $a_1,a_2,\dots,a_r$ such that $a_1 \le a_2 \le \dots \le a_r$.
For any $s \ge 1$, there exists a subset $\{a_{i_1},\dots,a_{i_l}\}$ of $\{a_1,\dots,a_r\}$ such that $a_{i_1} + \dots + a_{i_l} = s$.
This holds for $\{a_1,\dots,a_r\}$ if and only if $a_{i} - 1 \le A_{i-1}$ for any $i = 1,\dots,r$.
Indeed, if $a_{i} - 2 \ge a_1 + \dots a_{i-1}$ for some $i$, then there is no set of size $a_i - 1$ such that it is a union of components.
Conversely, if $a_{i} - 1 \le a_1 + \dots + a_{i-1}$ for all $i = 1,\dots,r$, we have $a_1 \le 1$, and by induction on $i$, there is a subset $\{a_{i_1},\dots,a_{i_l}\}$ of $\{a_1,\dots,a_i\}$ such that $a_{i_1} + \dots + a_{i_l} = s$ for any $1 \le s \le A_i$.
\end{proof}

By Proposition~\ref{prop: bh=n}, a graph $G$ with $\ti\alpha(G) = n$ has at least $\lceil\log_2 (n+1) \rceil$ components.
Indeed, by induction on $i$, we have $a_i \le 2^{i-1}$ for every $i = 1,\dots,r$, hence, $n = \sum\limits_{i=1}^r a_i \le \sum\limits_{i=1}^r 2^{i-1} = 2^r - 1$.

\section{Conclusion}

As an analog of Dirac's Theorem involving the bipartite-hole-number $\ti\alpha(G)$, McDiarmid and Yolov's Theorem~\ref{MY} implies Dirac's Theorem.
However, our Theorem \ref{2-conn} turns out to be independent of Ore's Theorem.
There exist $2$-connected graphs with $n \le \sigma_2 \le 2 \ti\alpha - 2$, and there also exist $2$-connected graphs with $2\ti\alpha - 1 \le \sigma_2 \le n-1$.
Both Dirac's Theorem and Ore's Theorem are special cases of a more general result involving a closure operator due to Bondy and Chv\'{a}tal \cite{BondyChvatal}.  The \emph{closure $\cl(G)$} of an $n$-vertex graph $G$ is obtained by repeatedly adding edges between pairs of nonadjacent vertices whose degree sum is at least $n$.  Both Dirac's Theorem and Ore's Theorem, as well as some other results, are special cases of the fact that $G$ is Hamiltonian if $\cl(G)$ is complete.  We also show that there exist $2$-connected graphs with $\sigma_2 \ge 2\ti\alpha-1$ whose closures are not complete.  Thus, there are graphs where Theorem \ref{2-conn} applies, but no conditions derived from having a complete closure apply.

Let $G \vee H$ denote the join of two graphs, obtained by adding an edge between every vertex of $G$ and every vertex of $H$.  For example, $K_a \oplus_1 K_b$ can also be regarded as $(K_{a-1} \cup K_{b-1}) \vee K_1$.

\begin{example}
\label{ex: 2-sum}
Let $G = (K_{a-2} \cup K_{n-a}) \vee K_2$ for $\frac n4 + 2 \le a \le \frac n2$.
Then, $\sigma_2(G) = (a-1) + (n+1-a) = n$ and $\ti\alpha(G) = \ti\alpha(K_{a-2} \cup K_{n-a}) =  \min\{2a-3, n-a+1\}$.
Clearly $G$ is $2$-connected, and the restrictions on $a$ mean that $\ti\alpha(G) = 2a+3$ and $n = \sigma_2 \le 2 \ti\alpha(G) - 2$.
\end{example}

\begin{example}
    \label{ex: cycle clique join}
    Let $G = (C_p \cup K_q) \vee K_r$ for positive integers $p,q,r$ such that $p \ge 4$, $q \ge 2p$ and $2p-1 \le r \le p+q-5$.
We have $n = p + q + r$, $\sigma_2(G) = 2(r + 2) = 2r + 4$ and $\ti\alpha(G) = 2p+1$ since there are no $(p+1,p+1)$-bipartite-holes and $G$ contains an $(s,2p+1-s)$-bipartite-hole for any $1\le s\le p$ by using a vertex subset of $C_p$ of size $s$ and a vertex subset $K_q$ of size $2p+1-s \le 2p \le q$.
Therefore, $n-1 = p+q+r -1\ge 2r + 4 = \sigma_2(G) \ge 2\ti\alpha(G) = 4p+2$.
\end{example}

We present another infinite family of graphs that are $2$-connected, regular and satisfy $n-p = \sigma_2 \ge 2 \ti\alpha(G) - 1$ for an arbitrary odd positive integer $p=2q-1$, $q \ge 1$. The Bondy-Chv\'{a}tal closure of these graphs is not complete.

\begin{example}\label{ex: circulant}
    Let $q \ge 1$ be an integer and $k \ge 2q$ be an integer.
    Let $d = 2(q+1)k = 2qk + 2k$ and $n = 2d + 2q-1 = 4qk + 4k + 2q-1$.
We define a graph $G_{q,k}$ whose vertex set is the group $\mZ_n = \{0,1, \dots, n-1\}$ with addition modulo $n$.  For each $i = 0, 1, \dots, q$ let $U_i = \{2ik+1, 2ik+2, \dots, 2ik+k = (2i+1)k\}$, let $U = \bigcup_{i=0}^q U_i$, and let $xy \in E(G_{q,k})$ if and only if $y \in x \pm U$.

Then $G= G_{q,k}$ is a $2$-connected $d$-regular graph on $n = 2d+2q-1$ vertices, and hence $\sigma_2 = 2d = n - (2q-1)$.  Since $n > 2d$, $\cl(G) = G$, which is not complete.

Showing that $\sigma_2(G) = 2d \ge 2\ti\alpha(G) - 1$ is equivalent to showing that $\ti\alpha(G) \le d$.
To this end, we claim that for any two vertices $x$ and $y$, 
$|N[x] \cup N[y]| \ge d + 2q+1$.
Without loss of generality, we may assume that $x = 0$, so that $N[x] = (-U) \cup \{0\} \cup U$, and $y \in \{1,\dots,\frac{n-1}{2}\}$, where $\frac{n-1}{2} = 2qk + 2k + q -1$.

First suppose that $y \ge k+1$.  
Consider the set of $2k+1$ consecutive vertices $Y = \{y-k, y-k+1, \dots, y+k\} \subseteq N[y]$.
Since the largest element of $U$ is $(2q+1)k = 2qk+k$,
every $z \in Y+U$ satisfies $z \le (\frac{n-1}{2}+k)+(2qk+k) = (2qk+2k+q-1)+2qk+2k = 4qk + 4k + q-1 < n$, so that $0 \notin Y+U$ and $Y \cap (-U) = \emptyset$.  Also $0 \notin Y$, and hence $Y \cap N[x] \subseteq U$.
If $Y$ contains elements of some $U_i$ and $U_{i+1}$ then $Y\setminus N[x]$, and hence $N[y]\setminus N[x]$, contains all $k$ points in between $U_i$ and $U_{i+1}$.
If $Y$ contains elements of at most one set $U_i$, then $Y\setminus N[x]$, and hence $N[y]\setminus N[x]$, has at least $k+1$ points.  In both cases, $|N[y]\setminus N[x]| \ge k \ge 2q$.

Now suppose that $1 \le y \le k$.
Then $\{k+1,3k+1,\dots, (2q+1)k+1\} \cup \{-2k,-4k,\dots,-2qk\} \subseteq N[y]\setminus N[x]$ and $|N[y] \setminus N[x]| \ge 2q+1$.

Thus, in all situations $|N[y] \setminus N[x]| \ge 2q$, and
hence $|N[x] \cup N[y]| = |N[x]| + |N[y] \setminus N[x]| \ge d+1 + 2q$.
For any vertex subset $T$ such that $(\{x,y\},T)$ is a bipartite-hole, $T\subseteq V(G)\setminus (N[x]\cup N[y])$ and $|T| \le n - (d + 2q+1) =  d-2$.
Hence, there is no $(2,d-1)$-bipartite-hole and $\ti\alpha(G_{q,k}) \le d$.
\end{example}

If $q = 1$, the graphs $G_{q,k}$ are $d$-regular graphs on $2d + 1$ vertices, which are known to be Hamiltonian by a result of Nash-Williams~\cite{Nash-Williams} (see also \cite{Bondy1978OreCE}).  However, when $q \ge 2$ we do not know of any results showing that $G_{q,k}$ is Hamiltonian except McDiarmid and Yolov's Theorem \ref{MY} or our Theorem \ref{2-conn}.

The Chv\'{a}tal-Erd\H{o}s Theorem~\cite{ChvatalErdos} says that a graph on at least three vertices is Hamiltonian if $\kappa \ge \alpha$.  Bondy~\cite{Bondy1978OreCE} showed that Ore's condition $\sigma_2 \ge n$ implies that $\kappa \ge \alpha$, and hence the Chv\'{a}tal-Erd\H{o}s Theorem implies Ore's Theorem.  We do not know whether this is also true of our Theorem~\ref{2-conn}.

\begin{question}
Does every $2$-connected graph $G$ satisfying $\sigma_2(G) \ge 2\ti\alpha(G)-1$ also satisfy $\kappa(G) \ge \alpha(G)$?  Any $G$ that satisfies the first condition but not the second condition must have $\sigma_2(G) < n$.
\end{question}

The graphs in Example \ref{ex: cycle clique join} have $\kappa(G) \ge r \ge p+1 \ge \alpha(G)$. Finding the connectivity and independence number of the graphs in Example \ref{ex: circulant} seems to be a nontrivial exercise, so we have not determined whether they have $\kappa \ge \alpha$.

\bibliographystyle{plain}

\bibliography{References}

\begin{thebibliography}{10}

\bibitem{Bondy1971pancyclic}
J.~A. Bondy.
\newblock Pancyclic graphs {I}.
\newblock {\em J. Combinatorial Theory Ser. B}, 11(1):80--84, 1971.

\bibitem{Bondy1973pancyclic}
J.~A. Bondy.
\newblock Pancyclic graphs: recent results.
\newblock In {\em Colloq. Math. Soc. J{\'a}nos Bolyai}, pages 181--187, 1973.

\bibitem{Bondy1978OreCE}
J.~A. Bondy.
\newblock A remark on two sufficient conditions for {H}amilton cycles.
\newblock {\em Discrete Math.}, 22(2):191--193, 1978.

\bibitem{BondyChvatal}
J.~A. Bondy and V.~Chv\'atal.
\newblock A method in graph theory.
\newblock {\em Discrete Math.}, 15(2):111--135, 1976.

\bibitem{BondyMurty}
J.~A. Bondy and U.~S.~R. Murty.
\newblock {\em {G}raph {T}heory}, volume 244 of {\em Graduate Texts in
  Mathematics}.
\newblock Springer, 2008.

\bibitem{ChvatalErdos}
V.~Chv\'{a}tal and P.~Erd\H{o}s.
\newblock A note on {H}amiltonian circuits.
\newblock {\em Discrete Math.}, 2:111--113, 1972.

\bibitem{Dirac}
G.~A. Dirac.
\newblock Some theorems on abstract graphs.
\newblock {\em Proc. London Math. Soc. (3)}, 2:69--81, 1952.

\bibitem{DMS}
Nemanja Draganić, David Munhá~Correia, and Benny Sudakov.
\newblock A generalization of {B}ondy’s pancyclicity theorem.
\newblock {\em Combin. Probab. Comput.}, 33(5):554–563, 2024.

\bibitem{Karp1972}
Richard~M. Karp.
\newblock Reducibility among combinatorial problems.
\newblock In {\em {C}omplexity of {C}omputer {C}omputations ({P}roc. {S}ympos.,
  {IBM} {T}homas {J}. {W}atson {R}es. {C}enter, {Y}orktown {H}eights, {N}.{Y}.,
  1972)}, The IBM Research Symposia Series, pages 85--103. Plenum, New
  York-London, 1972.

\bibitem{MY}
Colin McDiarmid and Nikola Yolov.
\newblock Hamilton cycles, minimum degree, and bipartite holes.
\newblock {\em J. Graph Theory}, 86(3):277--285, 2017.

\bibitem{MY-random-perfect-graph}
Colin McDiarmid and Nikola Yolov.
\newblock Random perfect graphs.
\newblock {\em Random Structures Algorithms}, 54(1):148--186, 2019.

\bibitem{Nash-Williams}
C.~St.~J.A. Nash-Williams.
\newblock Valency sequences which force graphs to have hamiltonian circuits.
\newblock {\em Interim report, University of Waterloo}, 1969.

\bibitem{Ore}
Oystein Ore.
\newblock Note on {H}amilton circuits.
\newblock {\em Amer. Math. Monthly}, 67:55, 1960.

\end{thebibliography}

\end{document}